\documentclass[11pt]{article}
\usepackage{latexsym,euscript,amsmath,amssymb,amsbsy,amsfonts,amsthm,amsopn,amstext,amsxtra,amscd}
\usepackage{epsfig}
\usepackage{graphics}
\usepackage{url}

\usepackage[english]{babel}

\theoremstyle{plain}
\newtheorem{theorem}{Theorem}

\theoremstyle{definition}

\newtheorem{point}{}

\newcommand{\N}{{\mathbb N}}

\newcommand{\Q}{{\mathbb Q}}

\newcommand{\II}{{\mathbb I}}
\newcommand{\Z}{{\mathbb Z}}

\newcommand{\D}{\text{$\mathcal{D}$}}

\def\T{T_{\boldsymbol{\pi}}}

\begin{document}
\centerline{\bf On one type of ud - preserving mapping}
$$ $$
\centerline{Milan Pa\v st\'eka}
$$ $$
{\bf Abstract.}{\it In this paper, we construct a class of mappings on the unit interval. These mappings preserve uniform distribution, and their iterations form a sequence that is Buck uniformly distributed. In the third part, we prove some properties of these mappings.}
\footnote{M a t h e m a t i c s S u b j e c t C l a s s i f i c a t i o n: Primary: 11B05; Secondary: 11J83.\newline K e y w o r d s: density, uniform distribution, reminder class, dynamical systems. \newline Research supported by the Grant VEGA 2/0119/23.}
\newline
{\bf Notation} \newline
$\N$  - the set of natural numbers. \newline
$\Q$ - the set of rational numberrs \newline
$r+(m)=\{n \in \N; n \equiv r \pmod{m} \}$. \newline
$\Z_m=\{0, \dots, m-1\}$. \newline
$v^{-1}(S)= \{n \in \N; v(n) \in S\}$. \newline
$|I|$ - length of interval $I$. \newline
$\pi^n = \pi\circ \pi \circ \dots \circ \pi$ - $n$ times composition of permutation $\pi: \Z_m \to \Z_m$, \newline
$T^n=T\circ T\circ \dots \circ T$ - $n$- times. \newline

\subsection{Introduction}
The motivation of the present paper is the results of \v S. Porubsk\'y,
O. Strauch and T. \v Sal\'at published in 1988 in the paper \cite{pss}.
The uniform distribution of sequences was first defined in 1916  by Hermann Weyl in the famous paper \cite{wey}.
 Later, this research was developed in several ways (see \cite{D-T}, \cite{SP}, \cite{K-N}). In the paper \cite{pss} mentioned above, the authors study the mappings $[0,1]\to [0,1]$ preserving uniform distribution, see below.

Weyl's concept of uniform distribution can be considered as ud with respect to asymptotic density. We will focus on a specific type of uniform distribution, namely, on uniform distribution with respect to the Buck measure density.

In 1946, R. C. Buck introduced "measure density" in the paper \cite{BUC}. Uniform distribution with respect to measure density is called Buck's uniform distribution (see \cite{p2})
Let $S \subset \N$. The value
$$
\mu^\ast(S)=
\inf \Big\{\sum_{j=1}^k \frac{1}{D_j}; S \subset \bigcup_{j=1}^k r_j+(D_j)\Big\}
$$
is called {\it Buck's measure density} of $S$. \newline
The system of sets
$\D_\mu=\{S \subset \N; \mu^\ast(S)+ \mu^\ast(\N \setminus S)=1 \}$ is an algebra of sets and its elements are called
{\it Buck measurable sets}. Clearly
\begin{equation}
\label{remind}
r+(m)\in \D_\mu \ \land \ \mu(r+(m))=\frac{1}{m}, r \in \Z, m\in \N.
\end{equation}
A sequence $\{v(n)\}, v(n)\in [0,1]$ is called {\it Buck uniformly distributed modulo $1$} if for each subinterval
$I \subset [0.1]$ there holds $v^{-1}(I)\in \D_\mu$ and
$\mu(v^{-1}(I)) = |I|$,
(see  \cite{p2}, \cite{p3}, \cite{p4}, \cite{P-T1}).

\subsection{Mapping}
Let $m_i, i=1,2, \dots$ be a sequence of relatively primes natural numbers. Put $B_0=1$ and
$$
B_j=\prod_{i=1}^j m_i.
$$
Put $\Z_m=\{0, \dots, m-1$ for given $m \in \N$.
Each real number $\alpha \in [0,1)$ can be expressed uniquely in the form of Cantor series
\begin{equation}
\label{cantor}
\alpha = \sum_{j=0}^\infty \frac{b_j(\alpha)}{B_{j+1}}, \ b_j(\alpha)\in \Z_{m_{j+1}}
\end{equation}
under the condition that: \newline
{\bf (c)} for infinitely many $j$ there holds
$b_j(\alpha)< m_{j+1}-1$. \newline
The equality
\begin{equation}
\label{expansion}
\sum_{j=0}^r \frac{b_j}{B_{j+1}}+\sum_{j=r+1}^\infty \frac{m_{j+1}-1}{B_{j+1}}= \sum_{j=0}^{r-1} \frac{b_j}{B_{j+1}} +\frac{b_r+1}{B_{r+1}}.
\end{equation}
provides that each estimation (\ref{cantor}) can be rewritten to a form which fulfills {\bf (c)}.
The proof of this equality is an easy calculation that takes into account $m_{j+1}=\frac{B_{j+1}}{B_j}$.
\begin{point}
\label{unicity} If $\alpha$ is an irrational number, then each of expansions of $\alpha$ in the form (\ref{cantor}) fulfills the condition (c).
\end{point}

Consider moreover a sequence of permutations $\pi_n : \Z_{m_n}\to \Z_{m_n},$ for $n=1,2,3, \dots$. Let us consider a "vector" of permutations
$$
\boldsymbol{\pi}=(\pi_1, \pi_2, \pi_3, \dots ).
$$
This allows define a mapping
$$
T_{\boldsymbol{\pi}}: [0,1] \to [0,1],
$$
where
$$
T_{\boldsymbol{\pi}}(\alpha)= \sum_{j=0}^\infty \frac{\pi_j(b_j(\alpha))}{B_{j+1}}.
$$
We get immediately
\begin{point} The mapping $T_{\boldsymbol{\pi}}$ is uniformly continuous on the interval $[0,1]$.
\end{point}

Now we formulate the main result. Denote
\begin{equation}
\label{suma}
T^{(n)}_{\boldsymbol{\pi}}(\alpha)=\sum_{j=0}^\infty \frac{\pi_j^n(b_j(\alpha))}{B_{j+1}}.
\end{equation}
\begin{theorem}
\label{iter} Suppose that the permutations $\pi_i$ are cyclic of length $m_i$. Let $\alpha \in [0.1)$, then the sequence of iterations $\{v(n)\}$, where $v(n)=\T^n(\alpha)$ is Buck's uniformly distributed modulo $1$.
\end{theorem}
Let us remark that with except of a countable set of $\alpha$
the value $\T^{(n)}(\alpha)$ coincides with $n$ th iteration $\T^n$ in the point $\alpha$.

Following obvious property will be useful for the proof:
\begin{point}
\label{krit1} Let $A_1, \dots, A_s \subset \N$ such disjoint sets that $A_1\cup \dots \cup A_s =\N$ and
$$
\mu^\ast(A_1)+ \dots +\mu^\ast(A_s) \le 1.
$$
Then these sets are Buck measurable.
\end{point}

The proof of the following criterion is trivial, also:
\begin{point}
\label{krit2} A sequence $\{v(n)\}, v(n)\in [0,1]$ is Buck uniformly distributed if and only if for each $\varepsilon >0$
such disjoint intervals $I_1,\dots, I_s$ exist that
$$
I_1 \cup \dots \cup I_s =[0,1), \ |I_j| < \varepsilon,
$$
and $v^{-1}(I_j) \in \D_\mu$, $\mu(v^{-1}(I_j))= |I_j|$.
\end{point}

An important argument in the proof will consist of the following connection between the division of the unit interval and the digits of the Cantor expansion. Put
$$
I_{j}^{(r)}= \Big[\frac{j}{B_r}, \frac{j+1}{B_r}\Big), \ j=0, \dots, B_r -1.
$$
\begin{point} For each $j\in \{0, \dots B_{r}-1\}$ such finite sequence $b_0 \in \Z_{B_1}, \dots, b_{r-1} \in \Z_{m_r}$ exits, that
$\alpha \in I_{j}^{(r)}$ if and only if
$b_i(\alpha)=b_i, i=0, \dots, r-1$, under condition (c).
\end{point}
In this case we say that the sequence $b_0, \dots, b_{r-1}$
is {\it associated } to the interval $I_{j}^{(r)}$.
Unfortunately, we can not provide that for the values $\pi_i^n(b_i(\alpha))$ that fulfill the condition (c). From the equalities \ref{cantor}, \ref{expansion} we get

\begin{point} a)
$$
T^n(\alpha)\in I^{(r)}_j \Rightarrow \pi^n(b_i(\alpha)) = b_i,
i=0, \dots, r-1
$$
or  \newline b)
$$
T^n(\alpha)\in I^{(r)}_j \Rightarrow \pi^n(b_i(\alpha)) = b_i,
i=0, \dots, r-2,
$$
$$
 \pi^n(b_{m_r-1}(\alpha))=b_{m_r-1}+1,
b_i = m_i-1, i\ge r.
$$
\end{point}

We continue with some properties on permutations.
For each cyclic permutation $\pi=(0,a_1, \dots, a_{m-1})$ on the set $\Z_m$ and $r, s \in \Z_m$ such uniquely given $k\in \Z_m$ exists that
\begin{equation}
\pi^n(r)=s \Longleftrightarrow n \equiv k \pmod{m}
\end{equation}
for $n \in \N$. By an application of Chinese reminder theorem we get:
\begin{point}
\label{simult} Let $\pi_1, \dots, \pi_j$ be cyclic permutations on $\Z_{m_i}$ of length $m_i$. Then for each $r_i, s_i \in \Z_{m_i}$ such $k_j\in \Z_{B_j}$ exists that
$$
\forall i=1, \dots, j; \pi^n_i(r_i)=s_i \Leftrightarrow n \equiv k_j \pmod{B_j}.
$$
\end{point}
This can also be written in the form
\begin{equation}
\label{class}
\{n \in \N; \forall i=1, \dots, j; \pi^n_i(r_i)=s_i\}=
r_j+(B_j).
\end{equation}
This yields:
\begin{point}
\label{zerod}Let $m_i, i=1,2,3, \dots $ be an infinite sequence of relatively prime natural numbers. Suppose that
$r_i, s_i, i=1,2,3, \dots $ be a sequences where
$s_i, r_i \in \Z_{m_i}$. Denote $A$ the set of all natural
$n$ that $\pi_i^n(r_i)=s_i, i=1,2,3, \dots $. Then $A \in \D_\mu$ and $\mu(A)=0$.
\end{point}
\begin{point} For each $x \in [0,1]$ we have
$v^{-1}(\{x\}) \in \D_\mu$ and $\mu(v^{-1}(\{x\}))=0$.
\end{point}

{\bf Proof of Theorem \ref{iter}.} Let us consider interval
$I_j^{(k)}$ where $0\le j \le B_k$. Let $b_0, \dots, b_{k-1}$ be a sequence associated to this interval. Application of {\bf \ref{simult}} gives that such $r_k$ exists, that
$$
\pi^n_i(b_i(\alpha))=b_i , i=1, \dots, k \Leftrightarrow
n \in r_k +(B_k),
$$
for $n \in \N$. If $v(n)$ contains in its Cantor's expansion infinitely many summands with nominator $\pi^n_i(b_i(\alpha))$ not exceeding the value $m_i-2$ then
\begin{equation}
\label{prva}
v(n) \in I^{(k)}_j \Longrightarrow n \in r_k+(B_k).
\end{equation}
In the other case, we have
$$
v(n) \in I^{(k)}_j \Longrightarrow
$$
\begin{equation}
\label{extrem}
\Longrightarrow \pi^n_i(b_i(\alpha))=b_i, i<r, \hat{} \pi^n_r(b_i(\alpha))=b_r-1, \pi^n_r(b_i(\alpha))=m_{r+1}-1.
\end{equation}
Thus
$$
v^{-1}(I^{(k)}_j ) \subset  r_k+(B_k) \cup S,
$$
Where $S$ is the set of all $n$ fulfilling the condition
(\ref{extrem}). From {\bf\ref{zerod}} we get $S \in \D_\mu$
and $\mu(S)=0$.
Thus $\mu^\ast(v^{-1}(I^{(k)}_j )) \le \frac{1}{B_k}$ and so
$$
\sum_{j=0}^{B_k-1}\mu^\ast(v^{-1}(I^{(k)}_j ))\le 1.
$$
From {\bf \ref{krit1}} we obtain $v^{-1}(I^{(k)}_j )\in \D_\mu$ and $\mu(v^{-1}(I^{(k)}_j ))= \frac{1}{B_k}$. From {\bf \ref{krit2})} we get the assertion. \qed

We have defined $\T^{(n)}$ because
the expression (\ref{cantor}) is not uniquely determined. For this reason, $\T^n$ does not coincide with this mapping every time. From {\bf\ref{unicity}} we get that expression (\ref{cantor}) fulfills condition {\bf (c)} if it is expression of irrational number. Put
$$
\II = [0,1] \setminus \bigcup_{n=1}\T^{-n}(\Q).
$$
\begin{theorem} For $\alpha \in \II$ we have
$\T^n(\alpha)= \T^{(n)}(\alpha)$.
\end{theorem}

In the paper \cite{pss} mentioned above, the authors introduce the following type of mappings: \newline
A mapping $T:[0,1] \to [0,1]$ is {\it u. d.
prserving} if each sequence $\{T(v(n))\}$ is uniformly distributed modulo 1 for when the sequence $\{v(n)\}$ is uniform distributed modulo 1. The criterion proved in \cite{pss} provides that for a continuous mapping $[0,1] \to [0,1]$, it suffices if one uniformly distributed sequence transforms to a uniformly distributed sequence to be a uniform distribution preserving. Considering that $\{\T(\T^n(\alpha))
\}= \{\T^{n+1}(\alpha)\}$ we get

\begin{point} $\T$ is uniform distribution preserving.
\end{point}

\subsection{Monotonicity and Derivative of $T_{\boldsymbol{\pi}}$}
\begin{theorem}
\label{monot}The function $T_{\boldsymbol{\pi}}$ is not monotone in any subinterval of positive length.
\end{theorem}

{\bf Proof.} It suffices to prove the assertion for each
interval $I_s^{(j)}$, $s=1,2,3, \dots $, $j=0,\dots , B_s-1$.
The permutation $\pi_{s+1}$ is no identic and there exists such $k_1, k_2, k_3, k_4 \in \Z_{m_{s+1}}$ that
\begin{equation}
\label{nerov1}
k_1 < k_2 \land \pi_{s+1}(k_1)< \pi_{s+1}(k_2)
\end{equation}
and
\begin{equation}
\label{nerov2}
k_3 < k_4 \land \pi_{s+1}(k_3)> \pi_{s+1}(k_4).
\end{equation}
Put $\alpha_i=\frac{j}{B_s}+\frac{k_i}{B_{s+1}}$, $i=1, \dots, 4$. Then from (\ref{nerov1} we get
$$
\alpha_1 < \alpha_2 \land T_{\boldsymbol{\pi}}(\alpha_1) < (\alpha_2)
$$
and
$$
\alpha_3 < \alpha_4 \land T_{\boldsymbol{\pi}}(\alpha_3) > (\alpha_4).
$$
And so $\T$ is no non decreasing, no non increasing in $I_s^{(j)}$. \qed

Consider the expansion (\ref{cantor}) of given
$\alpha$. Put in this case $b_k(\alpha)=a_k$. Denote by $\alpha_s$ such number that
$a_k=b_k(\alpha_s), s \neq k$, and
$b_s(\alpha_s)= \ell_s$. Clearly
$$
\alpha- \alpha_s =\frac{a_s- \ell_s}{B_{s+1}}.
$$
And
$$
\T(\alpha)- \T(\alpha_s) =\frac{\pi_s(a_s)- \pi_s(\ell_s)}{B_{s+1}}.
$$
Therefore
\begin{equation}
\label{podiel}
\frac{\T(\alpha)- \T(\alpha_s)}{\alpha-\alpha_s}=
\frac{\pi_s(a_s)- \pi_s(\ell_s)}{a_s-\ell_s}.
\end{equation}
From this, we can conclude:
\begin{point}
\label{integer}If $T_{{\boldsymbol \pi}}$ has derivative in some
$\alpha \in [0,1]$ then $T_{{\boldsymbol \pi}}(\alpha)$ is integer and its value is $\pi_s(a_s)- \pi_s(a_s-1)$ for sufficiently large $s$ where $a_s \neq 0$,  or $\pi_s(0)- \pi_s(1)$ if for infinitely many $s$ we have $a_s=0$.
\end{point}
\begin{theorem} Each subinterval $I \subset [0,1]$ contains a point where $\T$ has no derivative.
\end{theorem}

{\bf Proof.} Suppose that $\T$ has derivative in each point
of some interval $[x_1, x_2] \subset [0,1], X_1 < x_2$. If
$\T'(x_1) \neq \T'(x_2)$ then Darboux theorem provides that
 $\T'$ reach each value between $ \T'(x_1)$ and $\T'(x_2)$.  We get a contradiction with {\bf{\ref{integer}}}. Thus
$\T'(x_1)= \T'(x_2)$.  And so we have that $\T'$ is constant in each closed interval, thus it is monotone - a contradiction
with Theorem \ref{monot}. \qed
\begin{point}
\label{jedna} If such infinite set $S$ exists, that for each
sequence of $a_s, s \in S$ such sequence of $\ell_s, s \in S$ can be selected that the term of right hand side of (\ref{podiel}) converges to $1$ then $T_{{\boldsymbol \pi}}$ has not derivative in any point of $[0,1]$.
\end{point}
{\bf Proof.} Suppose that $\T$ has a derivative. Then
$\T'(\alpha)=1$ for $\alpha \in [0,1]$. Thus
$\T(x)=x$, $x \in [0,1]$ - a contradiction. \qed
\begin{point} The condition of {\bf \ref{jedna}} are fulfilled
if for infinitely many $s$ we have $\pi_s=(0123\dots m_s-1)$.
\end{point}

Department of Mathematics and Informatics, Faculty of Education, University of Trnava \newline
and \newline
Institute of Mathematics, Slovak Academia of Sciences

\centerline{SLOVAKIA}

\end{document}